\documentclass[12pt]{article}

\usepackage{graphicx,verbatim,array,multicol,courier}
\usepackage{amsmath,amssymb,amsthm,mathrsfs}
\usepackage{hyperref}
\usepackage{graphicx}
\usepackage{fullpage}
\usepackage{topcapt}
\usepackage{bm}
\usepackage{relsize}
\usepackage{natbib}
\usepackage{color}
\usepackage{mathdots}
\usepackage{mathtools} 
\usepackage{epstopdf} 
\usepackage{soul} 
\usepackage{booktabs} 
\usepackage{dsfont}
\usepackage{dsfont}
\usepackage{color}
\usepackage[T1]{fontenc}
\usepackage[utf8]{inputenc}
\usepackage{url}

\newcommand{\diff}{\mathrm{d}}

\newcommand{\hatl}{\hat{\lambda}}
\newcommand{\iidsim}{\overset{iid}\sim}
\newcommand{\indsim}{\overset{ind}\sim}

\newcommand{\numbCondMMLE}{{B.1}}
\newcommand{\numbCondPriorBound}{{B.2}}
\newcommand{\numbCondKL}{{B.3}}

\definecolor{navy}{rgb}{0,0,0.502}
\definecolor{brown}{rgb}{0.59, 0.29, 0.0}

\hypersetup{colorlinks,%
	citecolor=blue,%
	filecolor=green,%
	linkcolor=red,%
	urlcolor=violet,%
}

\newcommand{\Norm}{\mathcal{N}}

\theoremstyle{definition}
\newtheorem{defi}{Definition}[section]

\theoremstyle{plain}
\newtheorem{theo}{Theorem}[section]

\newtheorem{prop}[theo]{Proposition}
\newtheorem{cor}[theo]{Corollary}

\theoremstyle{definition} 
\newtheorem{example}{Example}[section]
\newtheorem{ex}[example]{Example}

\theoremstyle{remark}
\newtheorem{rem}[defi]{Remark}

\usepackage{enumitem,float}

\newlist{inparaenum}{enumerate}{2}
\setlist[inparaenum,1]{label=(\alph*)}
\setlist[inparaenum,2]{label=(\roman{inparaenumi}\emph{\alph*})}

\begin{document}

\title{Empirical Bayes 	in Bayesian learning: understanding a common practice}
\author{Stefano Rizzelli\\ \scalebox{0.9}{Department of Statistical Sciences, University of Padova} \\ 
	\\
	Judith Rousseau \\
	\scalebox{0.9}{CEREMADE, Universit\'e Paris Dauphine}\\
	\scalebox{0.9}{and Department of Statistics, University of Oxford} \\ 
	\\
Sonia 	Petrone \\
\scalebox{0.9}{Department of Decision Sciences, Bocconi University}}

\maketitle

\begin{abstract}
In applications of Bayesian procedures, {once a class of priors has been chosen, it may be tempting to fix the prior's hyperparameters from the data,} in an empirical Bayes (EB) fashion, usually by their maximum {marginal} likelihood estimates (MMLE).  This is a quite common but questionable practice, lacking a rigorous theoretical basis. We provide a theoretical framework where this form of EB is regarded as a computational strategy for approximating a genuine Bayesian posterior distribution and prove its general properties for parametric models. While computing the MMLE may still be demanding, we prove novel results that allow us to provide a simple proxy. These results establish the limit behavior of the MMLE in quite general settings, including both identifiable and non-identifiable models - specifically, overfitted mixture models - significantly filling a gap in the literature. Moreover, we study higher order merging, showing that, when not degenerate, the EB posterior approximates at a faster rate an oracle-Bayes posterior distribution based on the prior law that, within the given class of priors, expresses the most information on the true model's parameters. This is a faster approximation than classic Bernstein-von Mises results.  Our work provides formal content to common beliefs on this popular practice. 
\end{abstract}

\noindent%


\section{Introduction}
\label{sec:intro}
In applications of Bayesian procedures, even when the form of the prior is carefully specified, it may be delicate to fix the prior hyperparameters, so that, to bypass additional computational issues or for other reasons, it is tempting to fix them from the data, typically by their maximum marginal likelihood estimates (MMLE), obtaining a so-called \lq \lq empirical Bayes posterior distribution''. This approach is quite common but engenders controversy, and Bayesian motivations are admittedly nonrigorous; see e.g. \cite{bergerBerliner1986}, who also refer to it, from \cite{good1966}, as type II MLE. Moreover, while usually referred to as \lq \lq empirical Bayes'' (EB), this problem deviates from the classic EB setting as traced by \cite{robbins1956}. Here, the setting is Bayesian; given the model  $Y \mid \theta \sim p_\theta^{(n)}(\cdot)$, where $n$ denotes information in the data $Y$, like a sample size, one has chosen a class of prior distributions on $\theta$, typically in a parametric density form $\pi_\lambda$, $\lambda \in \Lambda$ and, in principle, would also elicit $\lambda$ based on prior information, e.g. from experts and domain knowledge. Yet, it is not uncommon to choose it empirically, which gives the empirical Bayes flavor. However, unlike the classic EB setting, there is no frequentist randomness of the model parameters across parallel experiments, thus no true latent distribution or, in other words, no true prior.   The prior density $\pi_\lambda$ formalizes the subject's information on the constant, although unknown, parameter $\theta$ - ultimately, on the data. As we will argue, the matter is comparing the information conveyed by different priors. To clean up confusion, we refer to this setting as Empirical Bayes in Bayes (EBIB) to distinguish it from the classic EB  approach \citep[e.g.,][]{efron19}, which assume parameters' variability across parallel experiments with a true hidden density.  Despite being a quite common practice, EBIB remains controversial,  and properties seem mostly given on a case-by-case basis. The aim of this paper is to dig further for methodological clarity and provide general theoretical properties. 

Data-dependent priors are used constantly by practitioners, either formally or informally. Among the various ways of choosing $\lambda$ from the data, the MMLE $\hat \lambda_n$ is often considered as a (more) principled approach, see for instance Chapter 8 in \cite{Robert94}. Indeed, it can be viewed as a comparison a posteriori of different priors - here, different choices of $\lambda$ - based on the resulting predictive performance on the data. When data are observed sequentially, i.e. $Y=(Y_1, \ldots, Y_n)$, a popular measure of predictive performance is the overall log-predictive score, $\sum_{i=1}^n \log q_\lambda(y_i \mid y_{1:i-1})=\log m_{\lambda}^{(n)}(y)$, where $q_\lambda(y_i \mid y_{1:i-1})$ is the predictive density of $Y_i$ given $Y_{1:i-1}= (y_1, \ldots, y_{i-1})$ computed at $y_i$. Hence maximizing the log predictive score is equivalent to maximizing the marginal likelihood $m_{\lambda}^{(n)}(y)$ (see e.g. \citealp{fong20}).  

In this paper we are not advocating the use of EBIB approaches in comparison, in particular, to a more genuine Bayesian approach that would express a hyperprior $H$ on $\lambda$ and build a hierarchical prior $\pi(\theta)= \int \pi_\lambda(\theta) \diff H(\lambda)$. The latter means to use a different prior, which can ultimately convey more structured information on the data. Yet, we do not deal, here, with the more fundamental problem of proposing the class of priors, which is the core of wide Bayesian literature.  Rather, we aim at better understanding the implications of EBIB procedures, since they are so commonly used in practice.  In fact, even when a hyperprior $H$ is used, it typically depends on some hyperparameter, which itself needs to be elicited and can thus be fixed in a data-dependent fashion. Our point  is that the EB posterior distribution could be motivated as an approximation, possibly computationally convenient, of a genuine Bayesian posterior. However, whether such an approximation holds is only partially known, and what Bayesian procedure EBIB would specifically approximate is an entirely open question. The common sense is that  EBIB gives a \lq \lq better'' choice of $\lambda$ than a possibly poor elicitation, that could be used, for example, as a default choice in softwares for Bayesian computations.  This would be a relevant practical use of EBIB, yet in what sense, if any, would this choice be \lq\lq better''?

Results by \cite{p+r+s2014} show that, if the true data generating process is $p_{\theta_0}^{(n)}$ and if $\sup_\lambda \pi_\lambda(\theta_0)< \infty$, then EBIB aims at tracking a {\em prior oracle} choice of $\lambda$, i.e. $\lambda^*({\theta_0}) = \arg\max_\lambda \pi_\lambda(\theta_0)$;  however, the important case of $\sup_\lambda \pi_\lambda(\theta_0)=+\infty$, which is typical of shrinkage phenomena, is not yet understood. Moreover, even in the regular case, i.e. $\sup_\lambda \pi_\lambda(\theta_0)< \infty$,  the EB posterior, too, verifies the Bernstein-von Mises property, and in fact merges with {\em any} Bayes posterior from a smooth prior, thus one cannot disentangle any \lq\lq better'' behavior to first order, calling for finer (higher-order asymptotic) comparisons. 
 
The \textit{irregular} (or degenerate) case, i.e. $\sup_\lambda \pi_\lambda(\theta_0)= \infty$, which is not studied in the literature, covers a wide range of situations for which the behaviour of the MMLE is a priori unclear. To illustrate this, consider the following simple but important example of Automatic Relevance Determination, as described for instance in \cite{Tipping2001}  in the context of relevance vector machines.   Consider the Bayesian regression  model $ {Y \mid \beta, \sigma^2} \sim \mathcal{N}_n(X \beta, \sigma^2 I_n)$, where $\beta=(\beta_1, \ldots, \beta_p)^T$ and, for short here, $\sigma^2$ is known, while the $\beta_j$ are given independent Gaussian priors,  $\beta_j \mid \tau_j \indsim \mathcal{N}(0, \tau_j)$. To speed up computations, \cite{Tipping2001} suggests to approximate a hierarchical approach where $\tau_j \stackrel{iid}{\sim} \mathcal {IG}(a,b)$ by setting the $\tau_j$ at their MMLEs $\hat{\tau}_j$, in an EB fashion, also motivating it in the sparse setup, i.e. when the true $\beta_{0,j}$ is zero for some $j$, by noting that in this case the resulting EB posterior is sharply peaked at zero (see also \citealp{Wipf2007} and  \citealp{Rudy2021}). Now, being the prior $\pi_{\lambda}=\prod_{j=1}^p \pi_{\tau_j}(\beta_j)$ (independent Gaussian densities, with $\lambda=(\tau_1, \cdots, \tau_p)$), as soon as a true $\beta_{0,j}=0$  one has   $\sup_\lambda \pi_\lambda(\beta)= \infty$, which is attained for any vector $\lambda$ 
such that $\tau_j=0$. What is the asymptotic behaviour of $\hatl_n$ in this case? If $p=1$ then $\sup_\lambda \prod_{j}\pi_{j}(\beta_{0,j}|\tau_j) = +\infty$ is attained at the singleton $ \tau_1 = 0$ leading to the point mass prior $\delta_{(0)}$, but, to our knowledge, it is not proved that $\hatl_n$ converges to 0 even in this most simple case. We answer this question in a general framework in Section \ref{sec:gene:identifiable}.  Moreover, although the EBIB posterior is computationally attractive, maximizing the marginal likelihood still requires (lengthy) numerical optimization.  Interestingly, our theoretical results of  Section \ref{sec:main-limit-results} suggest a simple proxy of the MMLE, which is here available in analytic closed form, namely $\hat{\beta}_{\text{OLS}}^2$. Table  \ref{fig:tableMMLE_mixreg} displays the maximum absolute approximation errors of the proxy-MMLEs $\hat{\tau}_{n,j}^{\text{proxy}}$ with respect to the MMLEs $\hat{\tau}_{n,j}$ for a sample of varying size $n$ (Table 1 in Section D.1.2 of the Supplement reports individual values). We return on this is Section \ref{sec:proxy}.  

\begin{table}[h]
\centering
\begin{tabular}{c ccc}
\hline
 & {\small$n=40$} & {\small$n=80$} & {\small$n=150$} \\
\hline
{\small$\max_{1 \leq j \leq 15} |\hat{\tau}_j - \hat{\tau}_j^{proxy}|$} & {\small0.036} & {\small0.012} & {\small0.007} \\
\hline
\end{tabular}
\caption{{\small Maximum absolute differences between MMLE $\hat{\lambda}_n{=(\hat{\tau}_1, \ldots, \hat{\tau}_{15})}$ 
and proxy-MMLE  $\hatl_{n}^{\text{proxy}}{=(\hat{\tau}_1^{\text{proxy}}, \ldots, \hat{\tau}_{15}^{\text{proxy}})}$.
The true regression parameters  are $\beta_{0,1}=\beta_{0,2}=1$, $\beta_{0,j}=0$, for $j=3, \ldots, 15$, while {$\sigma^2=1$ is known}.}}
\label{fig:tableMMLE_mixreg}
\end{table}
Continuing on the above example, a popular Bayesian alternative to the EB approach described above is to use continuous shrinkage priors, which have a hierarchical construction where  $ \tau_j^2 \iidsim h_\lambda$, $ j=1, \ldots,p$. Although the approach is now fully Bayesian, the hyperprior $h_\lambda$ again depends on hyperparameters $\lambda$ that need to be chosen; for instance using their  MMLE, as in the Bayesian LASSO \citep{parkCasella2008}. In this case, unless $\beta_{0,j}=0$ for all $j\leq p$, $\sup_\lambda \pi_\lambda (\beta_0)< \infty$ and using \cite{p+r+s2014}, the EB posterior is asymptotically Gaussian and equivalent to {\em any} posterior distribution associated to a smooth prior. However, for finite $n$ the resulting posterior distributions can be quite different, and it is of importance to understand which of them the EB posterior is specifically targeting. We answer this question in Section \ref{sec:merging}, in the general regular setup.

Finally, the above examples require the model to be identifiable, i.e. the true parameter $\theta_0$ to be uniquely defined. What is the behaviour of EBIB when the model is not identifiable is an open question. To address this question we consider mixture models, which are ubiquitous and widely used for Bayesian model-based clustering: 
\begin{equation}\label{model:mixt}
f_\theta(y) = \sum_{j=1}^K \, p_j \, g_{\gamma_j}(y), \quad \theta = (p_1, \cdots, p_K, \gamma_1, \cdots, \gamma_K), \quad \gamma_j \in \Gamma \subset \mathbb R^{d_0}.
\end{equation}
A typical prior on the mixture weights {$(p_1, \ldots, p_K)$} is the Dirichlet distribution  $\mathcal{D}(\lambda, \cdots, \lambda)$, $\lambda>0$.  The hyperparameter $\lambda$ plays a crucial role in governing sparsity of the mixture weights' configuration and the cluster allocation; see, e.g., \cite{r+m2011},  who for overfitted mixtures showed a phase-transition behaviour at $\lambda = d_0/2$, where $d_0$ is the dimension of $\gamma_j$, which suggests to choose $\lambda < d_0/2$.  Nevertheless, choosing $\lambda$ within the range $(0,d_0/2)$ remains challenging. 
We show in Section \ref{sec:mixtures} that we can adapt our results of Section \ref{sec:gene:identifiable} to this case, by noting that the relevant prior (playing the role of $\pi_\lambda$ in the setting of Section \ref{sec:gene:identifiable}) is the prior law induced by the Dirichlet distribution on the unknown clustering of the data,  and relate the asymptotic behaviour of $\hatl_n$ to the maximizer of $\pi_{\lambda, H_n} (K_0)$, where $\pi_{\lambda, H_n}$ is the marginal prior of the random number of clusters $H_n$ in $n$ observations drawn from $f_{0}$, with $K_0$ the true number of clusters in the population.

\subsection{Contributions and structure of the paper} \label{sec:contributions-new}
In this paper we  address the above questions in a general context from a frequentist perspective, i.e. assuming that there is a true probability law $P_{\theta_0}^{(n)}$. We first describe the asymptotic behaviour of  EBIB, from which we propose a proxy of the MMLE that is typically easy to compute.  We then look at asymptotic agreement of the EB and Bayesian posterior distributions in terms of strong frequentist merging \citep{ghoshRamamoorthi2003}.  To our knowledge, the first results on   generic asymptotic behaviour of the MMLE  and {merging properties of the EB posterior} were given by \cite{p+r+s2014} who were mostly concerned with regular parametric models $P_\theta^{(n)}$, with $\sup_\lambda \pi_\lambda(\theta_0)<\infty$.  Extensions to nonparametric models have been obtained in \cite{j+sz2017,zhang:gao20}, in the context where $\lambda$ is finite dimensional and drives the posterior contraction rates. In this paper we treat parametric models, possibly irregular (i.e. $\pi_\lambda(\theta_0)\in (0, +\infty]$) and possibly with increasing dimension, for which the \textit{signal} on $\lambda$ in $m_\lambda^{(n)}$ is not as strong (the posterior contraction rates remain the same for all $\lambda$).  

The limit behavior of the MMLE is studied in Section \ref{sec:MLE}. We prove that, under standard assumptions, {\em for any} $\theta_0$ the MMLE $\hatl_n$ converges to the {\em prior-oracle value} $\lambda^*(\theta_0)$ that maximizes the prior density at the true $\theta_0$. Remarkably, our limit results (Theorems \ref{theo:MMLE} and \ref{th:mixtures}) hold in a large variety of settings, not covered by \cite{p+r+s2014}, allowing for multidimensional hyperparameters and  both identifiable and non-identifiable models -- specifically, overfitted mixture models (Section \ref{sec:mixtures}). These results significantly fill a gap in the literature on MMLE and, importantly,  allow us to provide a cheap but sound proxy for the MMLE (Corollary \ref{cor:proxy}). Section \ref{sec:mixtures} is also of independent interest. Moreover, we clarify the conceptual framework underlying these results from a frequentist perspective. By noting that, in a genuine Bayesian approach, $m_\lambda^{(n)}(\cdot)$ is the {\em joint density} of $Y$  so that, formally, the class of models is  $\{m_\lambda^{(n)}, \lambda \in \Lambda\}$ and, when read as a function of $\lambda$, $m_\lambda^{(n)}$ is the {\em likelihood} -- while the model is commonly considered to be $\{p_\theta^{(n)}, \theta \in \Theta\}$ (even in Bayesian settings) and $m_\lambda^{(n)}$ read as the {\em marginal likelihood} -- we shed light on the nature of EBIB as a problem of {\em maximum likelihood} estimation under model misspecification - which seems to have been surprisingly overlooked in the literature, where ad-hoc terms such as {\em type II MLE} or MMLE are rather used.  {\em In this light}, it becomes natural to conjecture that {\em the MLE} $\hatl_n$ should converge to the pseudo-true value of $\lambda$. However, this is an unexplored case of model misspecification; to our knowledge, results are available only for the case where both the true and the misspecified densities admit some form of factorization, as in \cite{berk66, huber67, white82, douc12}, with the exception of \cite{Naulet23} in the context of a  specific multigraphex process, with partially explicit likelihood. The challenge and novelty  in our case is in the non-factorizable and non explicit nature of the misspecified model.  Moreover, in EBIB there is no true value of $\lambda$;  but we can give its {\em pseudo-true} value an interpretation  as a {\em KL-oracle} choice of prior hyperparameters, and provide its explicit form  by proving ({Theorem \ref{theo:KL}}) that, for a rather general class of priors and under mild regularity assumptions, any KL-oracle value is a {\em prior-oracle}.

On this basis, in Section \ref{sec:merging} we give higher-order merging results. We prove that, when $\sup_\lambda \pi_\lambda(\theta_0) < \infty$ and strong merging holds, the EB posterior distribution is a higher order approximation to the {oracle-Bayes} posterior within the given class of priors (Proposition\ref{prop:TVpost_gen}). As a by-product, this implies that the EB posterior is a faster approximation of the oracle $\pi_{\lambda^*(\theta_0)}(\,\cdot \, |\, y)$ than Bernstein-von Mises Gaussian approximation, and might be exploited, for instance, to compute approximate credible regions. We also prove results for predictive densities (Proposition \ref{prop: pred}).  Examples and extensions  are provided in Section \ref{sec:examples},  with a final discussion in Section \ref{sec:final}.  Proofs and technical details are deferred to the Supplement,  which also contains additional theoretical results on MMLE, further background on overfitted mixture models, and additional examples and numerical illustrations.

\section{Limit  behavior of the MMLE} \label{sec:MLE}

Let us first formalize the notation. We consider the standard Bayesian setting where conditionally on a random parameter $\theta$, taking values in $\Theta \subset \mathbb{R}^d$, the observations $Y$  have   probability law $P_\theta^{(n)}$  with  densities $p_\theta^{(n)}$ with respect to some measure {$\mu$}. Typically $Y = (Y_1, \ldots,Y_n)$ is a sample from an infinite sequence $(Y_i)_{i \geq 1}$, but it can be more general. The researcher chooses a class of prior distributions $\{\Pi_\lambda$, $\lambda \in \Lambda\}$ on $\Theta$ with densities $\pi_\lambda$ with respect to a measure $\nu$. The class of priors  is considered as given in our analysis. For a hyperparameter choice $\lambda$, the marginal density of $Y$ at $y$ is 
$$ 
m_\lambda^{(n)}(y) = \int p_\theta^{(n)}(y) \, \pi_\lambda(\theta) \diff\nu(\theta), 
$$ 
also referred to as the marginal likelihood, when seen as a function of $\lambda$. An {\em empirical Bayes} posterior density is obtained by plugging in a data-driven choice $\hatl_n \equiv \hatl_n(y)$, as  $\pi_{\hatl_n}(\theta \mid y) = p_\theta^{(n)}(y) \pi_{\hatl_n}(\theta) / m_{\hatl_n}^{(n)}(y). $ Typically, $\hatl_n$ is the MMLE of $\lambda$, obtained as 
\begin{equation}\label{eq:MLE}
\hat{\lambda}_n \in \arg \max_{\lambda \in \Lambda^{(n)}} m_\lambda^{(n)}(Y), \quad \Lambda^{(n)} \subseteq \Lambda.
\end{equation}
A more standard definition of the MMLE only refers to the case where $\Lambda^{(n)}=\Lambda$ and the term restricted MMLE is sometimes used when $\Lambda^{(n)}$ is a proper subset of $\Lambda$. Confining maximization of the marginal likelihood on a subset of $\Lambda$ may be desirable for practical reasons, see e.g. Section 2.1 of \cite{p+b20}, or  to ensure that a maximum exists, ruling out uninteresting regions of $\Lambda$. 

In this section, we study the asymptotic behavior of the MMLE under the true data probability law  $P_{0}^{(n)} = P_{\theta_0}^{(n)}$.  We first consider identifiable models; then, we focus on finite overfitted mixture models, which are strongly non-identifiable.

\subsection{General results for identifiable models} \label{sec:gene:identifiable}
We here consider identifiable models, where the true  density $p_0^{(n)} = p_{\theta_0}^{(n)}$ corresponds to a unique parameter $\theta_0 \in \Theta$.  In this case, \cite{p+r+s2014} showed that,  when $\sup_{\lambda \in \Lambda}\pi_{\lambda}(\theta_0)< \infty$, the MMLE $\hat \lambda_n$ converges to an {\em oracle set} $\Lambda^*(\theta_0)$ defined by
\begin{equation} \label{eq:oracle}
\Lambda^*(\theta_0)= \underset{\lambda} {\arg\max} \, \pi_\lambda(\theta_0) .
\end{equation}
They also conjectured, but did not prove, that a similar behavior could also hold for $\theta_0$ such that $\sup_{\lambda \in \Lambda}\pi_{\lambda}(\theta_0) = \infty$.  In this section we extend their results in three ways.  

Firstly, we give a much more articulated description of possible situations. As a basis, we consider a $d$-dimensional model's parameter $\theta$ and a structure of the prior distribution where the parameter vector $\theta$ is partitioned as $\theta = (\theta_1, \cdots, \theta_k)$, $k \leq d$, with a factorization of the prior density as 
\begin{equation}\label{eq:priorfactorization}
\pi_\lambda(\theta)=\prod_{j=1}^k \pi_{j}(\theta_j \mid \lambda_j), 
\end{equation}
with $\lambda_j \in \Lambda_j \subset \mathbb R ^{k_j}$.  Although necessarily not all-embracing, conditional independence structures of this nature are prevalent in Bayesian inference,  and allow to address the quite realistic situation where the  prior may explode for  some components of the parameter only - {which is not treated at all in \cite{p+r+s2014}}. As hinted in Section \ref{sec:intro}, this is typically the case when an EB selection of the prior hyperparameters is made at this stage; see, for instance, \cite{EB2019}, Sect. 1.1--1.2,  and \cite{neal2012}, Sect. 1.2.2--1.2.3 for further examples on generalized additive models and shallow neural networks; to name but a few. Specifically, we consider true values $\theta_0=(\theta_{0,1}, \ldots, \theta_{0,k})$ such that the prior explodes only for  some components $j$, say in a  subset $\mathcal{I}_\infty$ of $\{1, \ldots, k\}$ (for $\lambda_j$ on the boundary of $\Lambda_j$), while it is bounded otherwise; i.e., letting $q_{0,j}= \sup_{\lambda_j \in \Lambda_j} \pi_{\lambda_j }(\theta_{0,j}|\lambda_j)$,
\begin{equation}\label{eq:priorsetup}
q_{0,j}= \infty, \, j \in \mathcal{I}_\infty; \quad q_{0,j} <\infty, \, j \notin \mathcal{I}_\infty.
\end{equation}
The \lq\lq fully degenerate'' case,   $\sup_{\lambda_j} \pi_j(\theta_{0j}|\lambda_j)=\infty$ for all $j$, corresponds to $\mathcal{I}_\infty=\{1, \ldots,k\}$. The cases where $k>1$, and neither $\mathcal{I}_\infty$ nor $\mathcal{I}_\infty^c$ are empty, correspond to \lq\lq partially degenerate'' cases, which are conceptually and practically relevant. In the \lq\lq regular case'', $\mathcal{I}_\infty= \emptyset$. The latter is typically the case when, in a  \lq\lq fully Bayesian'' approach, a prior distribution $H_\gamma$ is assigned on $\lambda$, leading to a different joint prior  $\pi_\gamma(\theta)=\int \pi(\theta \mid \lambda) \, \diff H_\gamma(\lambda)$ on $\theta$, in turn indexed by hyperparameters that play the role of $\lambda$ in the above setting. Furthermore, the formulation \eqref{eq:priorfactorization} allows for high-dimensional models corresponding to allowing $k$  (and thus $d$) to grow with $n$, contrarywise to \cite{p+r+s2014}. For notational simplicity, we hereafter omit indexing $k$, $\mathcal{I}_\infty$ or $d$ by $n$.

Secondly, we study the convergence of the MMLE $\hatl_n=(\hatl_{n,1}, \ldots, \hatl_{n,k})$ to the set of maximizers of $\pi_j(\theta_{0j}|\lambda_j)$  for all components $j = 1, \ldots, k$, while existing results \citep{p+r+s2014} only cover the fully regular case where $\mathcal{I}_\infty =\emptyset$. Let us denote by $\Lambda_j^*{{(\theta_{0,j})}}$ the set of $\lambda_j^*\in \Bar{\Lambda}_j$, the closure of $\Lambda_j$,  that are maximizers of $\pi_j(\theta_{0j}|\lambda_j)$, i.e. such that  there exists $(u_m)_{m\geq 0}$  with $\pi_{j}(\theta_{0j}|u_m) \rightarrow q_{0j}$ and $\lambda^*_j = \lim_m u_m$.  We then define the  oracle set as
\begin{equation}    \label{eq:lambdaStar} 
\Lambda^* \equiv \Lambda^*(\theta_0) = \Lambda_1^*(\theta_{0,1}) \times \ldots \times \Lambda_k^*(\theta_{0,k}).
\end{equation}
We stress that, in fully or partially degenerate cases, the above set is smaller than the oracle set considered in \cite{p+r+s2014} -- see Example \ref{ex:mixedreg} below -- thus providing a much more precise notion of oracle hyperparameters and a more precise description of the asymptotic behaviour of $\hat \lambda_n$.

Thirdly, we show implications when EBIB is framed as a classic problem of {\em maximum likelihood}  estimation under model misspecification, as introduced in Section \ref{sec:contributions-new}. Conceptually, in this perspective the fully regular, degenerate and partially degenerate cases individuated above correspond, respectively,  to cases where the true density $p_{\theta_0}^{(n)}$ is outside the class of models ${\cal M}=\{ m_\lambda^{(n)}, \lambda \in \Lambda \}$, or may reach its boundary (for $\lambda$ approaching the boundary of $\Lambda$), or is outside the class $\cal M$ but, informally speaking, may reach the boundary for some components.

\subsubsection{ Convergence of the MMLE }  \label{sec:main-limit-results}
To keep the technical aspects to a minimum,  here and throughout the paper, detailed assumptions and proofs are omitted and provided in the Supplement. Let us define the semimetrics 
$$ 
d_j(\lambda_j , \lambda_j') = \left|\frac{\pi_j(\theta_{0j}|\lambda_j)}{1+\pi_j(\theta_{0j}|\lambda_j)} - \frac{\pi_j(\theta_{0j}|\lambda_j')}{1+\pi_j(\theta_{0j}|\lambda_j')}\right|
$$
on $\bar \Lambda_j$, and the semimetric $d(\lambda, \lambda') = \max_jd_j(\lambda_j, \lambda_j')$ on $\bar \Lambda$. This allows to  both handle possible non unicity of the maximizers of $\pi_j(\theta_{0j}|\lambda_j)$ and to  approach $q_{0j}=\infty$.  For any subsets $\Lambda', \Lambda''$ (possibly singletons), we denote $d(\Lambda', \Lambda'')= \inf_{\lambda' \in \Lambda', \lambda'' \in \Lambda''}d(\lambda', \lambda'')$.
\begin{theo}\label{theo:MMLE} 
Assume Conditions \numbCondMMLE-\numbCondPriorBound$\,$ in Section B.1 of the Supplement and assume that $\Lambda^{(n)}$ in \eqref{eq:MLE} is such that $\sup_{\lambda \in \Lambda}d(\lambda,\Lambda^{(n)})\to 0$. Then, as $n\to\infty$,  $d(\hat \lambda_n ,\Lambda^*) \to 0$ in probability under $P_{\theta_0}^{(n)}$, where $\Lambda^*$ is the prior-oracle set defined in (\ref{eq:lambdaStar}). \\
In particular, if  the set $\Lambda_j^*(\theta_{0,j})= \{ \lambda_j^*(\theta_{0,j})\}$ is a singleton  and {the map} $\lambda_j \mapsto d_j(\lambda_j, \lambda^*_j(\theta_{0,j})) $ has continuous inverse at $0$, then also $\hatl_{n,j} \to \lambda_j^*(\theta_{0,j})$ in the Euclidian metric.
\end{theo}
As commented earlier, this result substantially extends the scope of Theorem 1 in \cite{p+r+s2014}. The proof is provided in Section C.3 of the Supplement. 
Therein, Assumption \numbCondMMLE$\,$ allows to control the (penalized) likelihood ratio away from $\theta_0$ and Assumption \numbCondPriorBound$\,$ controls $\pi_\lambda(\theta)$ for $\theta$ near $\theta_0$.  It is satisfied in particular by regular parametric models, but can hold also when $k$ goes to infinity. Detailed comments are provided in Section B.1 in the Supplement. It is worth noting that, differently from Theorem 1 in \cite{p+r+s2014}, the proof uses carefully devised techniques for handling  prior densities near $\theta_0$, in order to cope with degeneracies, and for allowing the dimension of $\theta$ to increase.

\subsubsection{ Convergence of the  KL minimizers  } 
The following result, which is of some independent interest, takes the KL perspective regarding $\hatl_n$ as {\em the  MLE} under model misspecification, and clarifies the nature of the KL pseudo-true value. Informally, it shows that, for the large variety of cases delineated above,  any KL minimizer is also asymptotically a prior-oracle value in the sense of \eqref{eq:oracle}.
\begin{theo}\label{theo:KL}
	Assume Conditions  \numbCondPriorBound$\,$ and \numbCondKL$\,$ in Section B.1 of the Supplement.  Let $(\lambda_n^*)$ be any sequence of minimizers
    of $KL(p_{\theta_0}^{(n)}\Vert m_\lambda^{(n)})$ over    $\Lambda$, i.e. satisfying for every $n$
	\begin{equation}\label{eq:asympmin}
    \lim_{\lambda \to \lambda_n^*} 
		KL(	p_{\theta_0}^{{(n)}}
	\Vert 
    m_{\lambda}^{(n)})
	= {\inf_{\lambda \in {\Lambda}}}	KL( p_{\theta_0}^{{(n)}}
    \Vert	m_{\lambda}^{(n)}
	).
 \end{equation}
Then, $d(\lambda^*_n, \Lambda^*) \to 0$ as $n \to \infty$, where $\Lambda^*$ is the prior-oracle set defined in \eqref{eq:lambdaStar}. 

Moreover, if the set $\Lambda_j^*(\theta_{0,j})= \{ \lambda_j^*(\theta_{0,j})\}$ is a singleton  and $\lambda_j \mapsto d_j(\lambda_j, \lambda^*_j(\theta_{0,j})) $ has continuous inverse at $0$, then also $\lambda_{n,j}^* \to \lambda_j^*(\theta_{0,j})$ in the Euclidian metric.

\end{theo}
The proof is provided in Section C.1 of the Supplement. This result is consistent with Theorem 2.1 of \cite{c+b90}, {which is} given in the different context of Jeffrey's priors but allows to draw a similar conclusion in the case of i.i.d. data for strongly regular models of fixed finite dimension, and smooth prior densities, with the restriction that $\sup_{\lambda \in \Lambda}\pi_\lambda(\theta_0)<\infty$. We stress that our result is more widely applicable, covering, e.g., the case of non-identically distributed observations (such as in fixed design regression) or dependent data (such as finite Markov chains, see Section D.2 in the Supplement), and also holds for high-dimensional models and (partially) degenerate cases. To handle such a higher level of generality, we rely on assumption \numbCondKL$\,$ (Section B.1 of the Supplement), keeping to a minimum standard conditions in frequentist asymptotics for Bayesian nonparametric posteriors \citep[e.g.,][]{bnp2017}. In particular,  it requires sufficient prior mass on sets where the KL divergence $K_n(\theta)= E_{\theta_0}^{(n)}[\log p_{\theta_0}^{(n)}(Y)-\log p_\theta^{(n)}(Y)]$ and $p$-th variation $V_{n,p}(\theta):=E_{\theta_0}^{(n)}\left[ \log p_{\theta_0}^{(n)}(Y)-\log p_{\theta}^{(n)}(Y)-K_n(\theta)\right]^p$ are not \lq\lq too large\rq\rq. 

The following corollary pinpoints the consequent different scenarios for KL limiting values; namely,  (a)  $\mathcal{I}_\infty \neq \{1,\ldots,k\}$;  the true model is outside the (\lq misspecified') class ${\cal M}=\{m_\lambda^{(n)}, \, \lambda \in \Lambda \}$,   (b) $\mathcal{I}_\infty=\{1,\ldots,k\}$; the true model may be on the boundary of $\cal M$. 

\begin{cor}\label{cor:miss}
	Assume Conditions \numbCondPriorBound$\,$ and \numbCondKL$\,$ in Section B.1 of the Supplement hold true and let $\lambda_n^*$ be  any sequence of KL minimizers in the sense of  \eqref{eq:asympmin}.
	\begin{inparaenum}
    	\item \label{res:mult} If $\mathcal{I}_\infty \neq \{
        1,\ldots,k \}$ and, for some  $j \in \mathcal{I}_\infty^c$, there exist $\eta \in (0,1)$ and $\epsilon>0$ so that $\pi_j(\theta_j|\lambda_j)$ is uniformly bounded for $\Vert \theta_j -\theta_{0,j}\Vert<\epsilon$ and $\lambda_j$ for which $\pi_j(\theta_{0,j}|\lambda_j) >(1-\eta)q_{0,j}$, then	$\liminf_{n \to \infty} \lim_{\lambda \mapsto \lambda_n^*}KL(	p_{\theta_0}^{{(n)}}\Vert 	m_{\lambda}^{(n)}) >0.	$

		\item \label{res:sing} If $\mathcal{I}_\infty=\{1, \ldots,k\}$
        (fully degenerate case), $\Pi_\lambda$ converges weakly to $\delta_{(\theta_{0})}$ as $\lambda \to \lambda^*$, for all $\lambda^* \in \Lambda^*$, and $K_n(\theta)$, $V_{n,p}(\theta)$ are continuous at $\theta_0$ for every $n$, then  $\lim_{\lambda \mapsto \lambda_n^*}KL(	p_{\theta_0}^{{(n)}}
		\Vert m_{\lambda}^{(n)})=0$   for every $n$. In particular, this is the case if $\lambda_n^* \in \Lambda^*$.

	\end{inparaenum}
\end{cor}
The proof of Corollary \ref{cor:miss} is provided in Section C.2 of the supplement.

 \noindent In combination with Theorem \ref{theo:MMLE}, Theorem \ref{theo:KL}  has the following interesting implication.
\begin{cor}\label{cor:MMLE_kl}
    Under the assumptions of Theorem \ref{theo:MMLE} and the additional Condition \numbCondKL, if each set $\Lambda_j^*(\theta_{0,j})$ consists of a single point $\lambda_j^*(\theta_{0,j})$, then $d(\hatl_n,  \lambda^*_n) \to 0$  in probability under $P_{\theta_0}^{(n)}$, for any sequence $(\lambda_n^*)$  of  KL minimizers as in \eqref{eq:asympmin}.
\end{cor}

In light of Corollary \ref{cor:miss}, the convergence result above implies that, when there is no regular component in the prior and $\sup_\lambda \pi_\lambda(\theta_0)= \infty$ occurs because the prior degenerates at $\theta_0$, i.e. $\Pi_{\lambda^*{(\theta_0)}} = \delta_{(\theta_0)}$,  then the MMLE approaches the hyperparameter choice for which the KL information loss is null (Corollary \ref{cor:miss}\ref{res:sing}); more subtly, whenever $\mathcal{I}_\infty \neq \{1, \ldots,k\}$, the minimum KL information loss with MMLE remains positive (Corollary \ref{cor:miss}\ref{res:mult}).

\subsubsection{Proxy for the MMLE} \label{sec:proxy}
We highlight a practical implication of the limit results of this section, as given in Corollary \ref{cor:MMLE_kl}, for the construction of simple and computationally convenient proxies of the MMLE.

\begin{cor}\label{cor:proxy}
 Under the assumptions of Theorem \ref{theo:MMLE} (where in Condition $\numbCondPriorBound$ $\epsilon_n$ is replaced by an arbitrarily small $\epsilon>0$),  for any estimator $\hat{\theta}_n =(\hat{\theta}_{n,1}, \ldots,\hat{\theta}_{n,k})$ converging to $\theta_0$  in $P_{\theta_0}^{(n)}$-probability and with $\lambda^*(\hat{\theta}_n)=(\lambda_1^*(\hat{\theta}_{n,1}), \ldots, \lambda_k^*(\hat{\theta}_{n,k}))$ where $\lambda^*_j(\hat{\theta}_{n,j}) \in \arg \max_{\lambda_j} \pi_{j}(\hat{\theta}_{n,j}|\lambda_j)$ for $1 \leq j \leq k$, it holds that $d(\lambda^*(\hat \theta_n), \hat \lambda_n) \to 0$  in $P_{\theta_0}^{(n)}$-probability. 
\end{cor}

Similarly to before, if $\Lambda^*$ is a singleton at $\lambda^*(\theta_0)$, then the MMLE $\hat \lambda_n$ and the proxy-MMLE $\hatl_n^{\text{proxy}} \equiv$ $\lambda^*(\hat \theta_n)$ are close to each others in the sense that they both converge to the oracle $\lambda^*(\theta_0)$.  Corollary \ref{cor:proxy} holds for fixed dimension $k$. When $k$ goes to infinity and if,  for all $j$, $\|\hat \theta_{n,j} - \theta_{0j}\| \to 0$ in $P_{\theta_0}^{(n)}$-probability, then  $d_j(\lambda^*_j(\hat \theta_n), \hat \lambda_{n,j}) \to 0$ in $P_{\theta_0}^{(n)}$-probability for all $j$, and convergence also holds in standard Euclidean metric under the above-mentioned continuity conditions (Theorems \ref{theo:MMLE}--\ref{theo:KL}). As we will illustrate in Section \ref{sec:examples}, the expression of $\lambda^*(\theta) $  is often available in closed form, making the computation of $\hatl_n^{\text{proxy}}$ straightforward; even when it is not, computing $\hatl_n^{\text{proxy}}$ by maximizing the prior components $\pi_j(\hat{\theta}_{n,j}|\lambda_j)$ is typically much less demanding than maximizing the marginal likelihood. 

\begin{ex}  \label{ex:mixedreg} 
Consider the linear regression model described in Section \ref{sec:intro}, now letting $\sigma^2 \sim \pi$, a fully specified continuous and positive density on $\mathbb R_+$. The choice of the hyperparameters $\tau_j^2$ controls shrinkage, and in an EB fashion one would use their MMLEs. In the notation of \eqref{eq:priorsetup}, the true vector of coefficients $(\beta_0, \sigma_0)$ is partitioned as  $(\theta_{0,1}, \ldots, \theta_{0,k})$ where  $p=k$ and, without loss of generality, $\theta_{0,1}=(\beta_{0,1}, \sigma_0)$ and $\theta_{0,j}=\beta_{0,j}$, $2 \leq j \leq p$. We have an analogous partition of the hyperparameter vector  as $\lambda=(\lambda_1, \lambda_2, \ldots,  \lambda_{k})$ where $\lambda_j = \tau_j$, $1\leq j \leq p$. Moreover, the prior density can be factored into $k$ components $\pi_1(\theta_1|\lambda_1)= \mathcal{N}(\theta_{1};0, \tau_{1})\pi(\sigma^2)$ and $\pi_j(\theta_j|\lambda_j)=\mathcal{N}(\beta_{j};0,\tau_{j})$ for $2 \leq j \leq k$, which comply with \eqref{eq:priorsetup} with $\mathcal{I}_\infty=\{j: \, \beta_{0,j}=0 \}$. It is easy to verify that the unique maximizers of $\pi_\lambda(\theta_0)$ are $\lambda^*_{j}(\theta_{0,{j}})=\beta_{0,j}^2$ for $j \notin \mathcal{I}_\infty$ and $\lambda^*_{j}(\theta_{0,{j}})=0$ if $j \in \mathcal{I}_\infty$,  so that $\Lambda_j^*(\theta_{0,j})=\{\lambda^*_j(\theta_{0,j}) \}$ and the oracle prior  on the null beta's is degenerate on zero. We point out that the oracle set envisaged in \cite{p+r+s2014} is  much wider, including all hyperparameter vectors $\lambda$ such that $\lambda_j \in (0,\infty)$ for all $j \notin \mathcal{I}_\infty$ and  \textit{at least} one component $\lambda_j$, with $j \in \mathcal{I}_\infty$, is null.

The conditions of Corollary \ref{cor:MMLE_kl} are verified in Section D.1 of the Supplement, under suitable assumptions on the design matrix as soon as the dimension satisfies $p = o(n/\log n)$. Therefore,  both the MMLE $\hat \lambda_{n}$ and the KL  minimizers $\lambda_n^*$  approach the prior  oracle vector $\lambda^*(\beta_0)$. Table \ref{fig:tableMMLE_mixreg}, anticipated in Section \ref{sec:intro}, shows a good approximation of the MMLE through the proxy-MMLE $\hatl_{n}^{\text{proxy}}=\hat{\beta}_{\text{OLS}}^2$ (there, $\sigma^2$ was known).  For the same setting, Tables 2 and 3 of the Supplement also report the mean absolute errors $|\hat{\tau}_{n,j}-\hat{\tau}_{n, j}^{\text{proxy}}|$ and the empirical coverage of posterior credible intervals for $\beta_j$'s computed over 1000 samples. The MMLE is obtained through the sequence of updates from MacKay \citep[see][eq. (16)]{Tipping2001}. If $\sigma^2$ is unknown, the MMLE may be computed through the algorithm by \cite{dossLinero2024} (Sect. 2.2, pp. 605--607), leveraging on our proposed proxy MMLE for centering the required auxiliary hyperprior. 
\end{ex}

\subsection{Finite overfitted mixtures}\label{sec:mixtures}

In the previous section, we obtained results under the identifiability assumption that $P_0^{(n)}= P_{\theta_0}^{(n)}$ for a unique $\theta_0$.  For non-identifiable models, the prior-oracle set $\Lambda^*$  of the maximizers (in $\lambda$) of $\pi_{\lambda}(\theta_0)$ is not well defined. However, it might still be possible to define a prior-oracle set associated to the limiting behaviour of $\hat \lambda_n$ and $\lambda_n^*$. To illustrate this, we consider the important example of overfitted mixture models and we show that if $\lambda$ parametrizes the Dirichlet prior on the weights of the mixture, $\hat \lambda_n$ has a well-defined asymptotic behaviour, which is strongly related to a prior-oracle on the clustering structure.

More precisely, consider a mixture model as in \eqref{model:mixt}.  Assume, as in \cite{r+m2011}, that the true distribution is 
\begin{equation}\label{true:mixt}
f_0(y) = \sum_{j=1}^{K_0} \, p_j^o \; g_{\gamma_j^o}(y), \quad  K_0 < K.
\end{equation}
Then the model is not identifiable (even up to label switching) since $f_0$ can be written as $f_{\theta_0}$ with $\theta_0$ given by either emptying or merging the extra components. Namely, the true density $f_0=f_{\theta_0}$ can be reached by any $\theta_0$ of the form $\theta_0=( (p_j^o, \gamma_j^o), j=1, \ldots, K_0; \; p_j=0, j=K_0+1, \ldots, K)$ (emptying components) or $\theta_0= ((\gamma_i = \gamma_j^o, \; \sum_{i\in I_j}p_i=p_j^o), i \in I_j, \; j\leq K_0)$  for any partition of the mixture's indexes $i=1, \ldots, K$ in groups $I_j$ such that $\theta_0$ leads to the true $f_0$ (merging components). Thus, all the parameters $\theta_0$ above verify $f_0 = f_{\theta_0}$. In \cite{r+m2011} it is shown that if the prior on $\theta$ has the following form:
\begin{equation}\label{prior:mixt}
(p_1, \cdots, p_K) \sim \mathcal D(\lambda, \cdots, \lambda) \quad \mbox{and, independently,} \quad \gamma_j\stackrel{iid}{\sim} \pi_\gamma, \quad j=1, \ldots, K, 
\end{equation}
with $\lambda>0$, and under regularity assumptions on the model, then if $\lambda<d_0/2$, where $d_0$ is defined in \eqref{model:mixt}, the posterior distribution concentrates on the emptying extra-components configuration,  while, if $\lambda>d_0/2$, then the posterior concentrates on the merging-components configuration. As noted in Section \ref{sec:intro}, the choice of $\lambda$ still remains an imprecise issue since for finite $n$ it is not clear which value in the range $0 < \lambda < d_0/2$ should be chosen. It is therefore of interest to investigate if selecting $\lambda$ by maximizing the marginal likelihood would provide a reasonable answer, notably since $\lambda$  governs  the prior on the underlying clustering structure  of the data. The Dirichlet distribution on the mixture weights induces a probability law on such a clustering structure (or random partition), and in particular on the random number $H_n$ of nonempty clusters in $n$ observations from $f_\theta$: 
\begin{equation}    \label{eq:Hnprior}
\Pi_{\lambda}(H_n = h) = \frac{K!}{(K-h)!}\frac{\Gamma(\lambda K)}{\Gamma(\lambda K +n) } \mathscr{C}(n,h;\lambda),
\end{equation}
where $\mathscr{C}(n,h;\lambda)=(1/h!) \sum_{j=1}^h {h \choose j} (-1)^{h-j} \Gamma(n+\lambda j)/\Gamma(\lambda j)$ (see Theorem 3 of \citealp{l+p+r2020}). It can be easily seen that this prior law concentrates the unit mass on  $H_n=1$ for $\lambda \rightarrow 0$,  while for any fixed $\lambda>0$ it concentrates around $H_n = K$ - which is clearly undesirable if the mixture is overfitted. Thus, even restricting the selection to $(0, d_0/2)$, the issue is not only to select a value of $\lambda$ in that range, but to study whether EBIB by MMLE may give a \lq\lq well-informed'' selection $\hatl_n$ tracking an oracle sequence that increasingly penalizes an excessive number of clusters; intuitively, that converges to zero at an appropriate rate - not too fast (as $\lambda=0$ gives a trivial one-cluster configuration), but favoring the true number $K_0$ of clusters.

We show below in Theorem \ref{th:mixtures} that indeed  $\hat{\lambda}_n = O_{P_0}( \log \log n/ \log n) $ and  that  the Kullback-Leibler minimizer $\lambda_n^*=O(\log \log n/\log n)$, they therefore both have a similar asymptotic behaviour. Interestingly, the maximizer in $\lambda$ of $\Pi_{\lambda}(H_n = K_0)$ is asymptotically equivalent to $\lambda_n (K_0):= c^*/\log n$ with $c^* = \log (K-1) - \log (K-K_0)$ for $K_0>1$ and $\lambda_n(K_0) = 0$ for $K_0=1$. We therefore show that $\hat \lambda_n $ and $\lambda_n^*$ are asymptotically close to the maximizer of $\Pi_\lambda (H_n = K_0)$.  Since the clustering structure of the data  is of interest in many contexts, Theorem \ref{th:mixtures} has interest in its own right. It is however still unclear if the EBIB posterior on $H_n$ concentrates on $K_0$. 

\begin{theo}\label{th:mixtures} 
Consider data $Y = (Y_i, i=1, \ldots, n)$ with $Y_i$ i.i.d. from $P_0$ with density \eqref{true:mixt} and consider the model \eqref{model:mixt} with prior of the form \eqref{prior:mixt}. Assume in  addition that the component's density $g_\gamma$  verifies the smoothness and strong identifiability assumptions of \cite{r+m2011}, as recalled in Section C.4.1 of the Supplement. If $\pi_\gamma$  is positive and continuous on $\Gamma$, then for any $M>0$
\begin{align*} 
\hat \lambda_n &= \text{\emph{argmax}}_{\lambda \in (0, M]} m_\lambda^{(n)}(Y) = O_{P_0}\left( \frac{ \log \log n }{ \log n} \right),\\
 \lambda_n^* &= \text{\emph{argmin}}_{\lambda \in (0,  M]} KL(p_0^{(n)} \Vert m_\lambda^{(n)}) = O\left( \frac{ \log \log n }{ \log n} \right).
 \end{align*}
\end{theo}

Theorem \ref{th:mixtures} shows that both  the MMLE and the Kullback-Leibler minimizers are of order $O(\log \log n/\log n)$. A full version of Theorem \ref{th:mixtures} is provided in Section C.4.1 of the Supplement, where optimization over $\lambda$ is allowed to be over $(0, L_n)$ where $L_n$ grows to infinity at a certain rate. This shows that the MMLE belongs to the range $(0,d_0/2)$ yielding configurations where the extra states have very small posterior probability, thus allowing for interpretable parametric inference. In fact, given the result of \cite{r+m2011}, it makes sense to directly compute the restricted MMLE over $(0, d_0/2)$.

\begin{ex} Consider an overfitted bivariate Gaussian location mixture  with $K=10$ components and known diagonal covariance matrix, and let us compare different choices of $\lambda$. Figure \ref{fig:mix} shows the log marginal likelihood of $\lambda$ (left-hand panel) and the priors on $H_n$ (right-hand panel) with $\lambda=1$, corresponding to a uniform prior on the mixture weights, often used in common practice;  $\lambda$ equal to MMLE $\hatl_n= \arg\max_{\lambda \in (0,1)}m_\lambda^{(n)}(Y)$ and a resticted version, $\hatl_n^{\text{res}} = \arg\max_{\lambda \in \Lambda^{(n)}}m_\lambda^{(n)}(Y)$ with grid $\Lambda^{(n)}= \{ \log\left(\frac{K-1}{K-h}\right)/\log n, \, h=2, \ldots, K-1 \}$ motivated by the formula of $\lambda_n(K_0)$;  $\lambda$ equal to a proxy-MMLE $\lambda_n(\hat{K}_n)$ that uses the consistent estimator of $K_0$  with smoothly clipped absolute deviation penalty \citep{manole21}. The displayed results are based on  a synthetic random sample of size $n=400$ from $f_{0}= 0.5 \mathcal{N}_2(\mu_{0,1}, \text{diag}(1,1))+0.5 \mathcal{N}_2(\mu_{0,2}, \text{diag}(1,1))$, with $\mu_{0,1}=(0,0)^T$, $\mu_{0,2}=(5,5)^T$ and $K_0=2$. The prior distribution used for  emission parameters is $ \mu_j \overset{iid}{\sim} \mathcal{N}( \xi, \text{diag}(\tau_1, \tau_2 ))$, with $\xi\approxeq (2.57, 2.48)^T$ and $\tau_i\approxeq 0.25$, $i=1,2$.

Although the MMLE, its restricted version and the proxy-MMLE we obtained are in the range $(0, d_0/2)$ (here $d_0/2=1$) the proxy-MMLE $\approxeq 0.0197$ equals  the $H_n$-prior oracle $\lambda_n(K_0)$ and induces a prior mode at the true number of clusters $K_0$; while the restricted MMLE is larger, $\hat{\lambda}_n^{\text{res}} \approxeq 0.367$, thus yielding a prior on $H_n$ which favours more clusters, yet much less than the prior with ordinary MMLE $\hatl_n \approxeq 0.63$ and with default choice $\lambda=1$. This suggests a slower convergence of the MMLE $\hatl_n$ to 0 than $\lambda_n(K_0)$ and that the rate obtained in Theorem \ref{th:mixtures} may be sharp.
\begin{figure}[h]   
\label{fig:mix}
    \centering
    \includegraphics[width=0.49\linewidth]{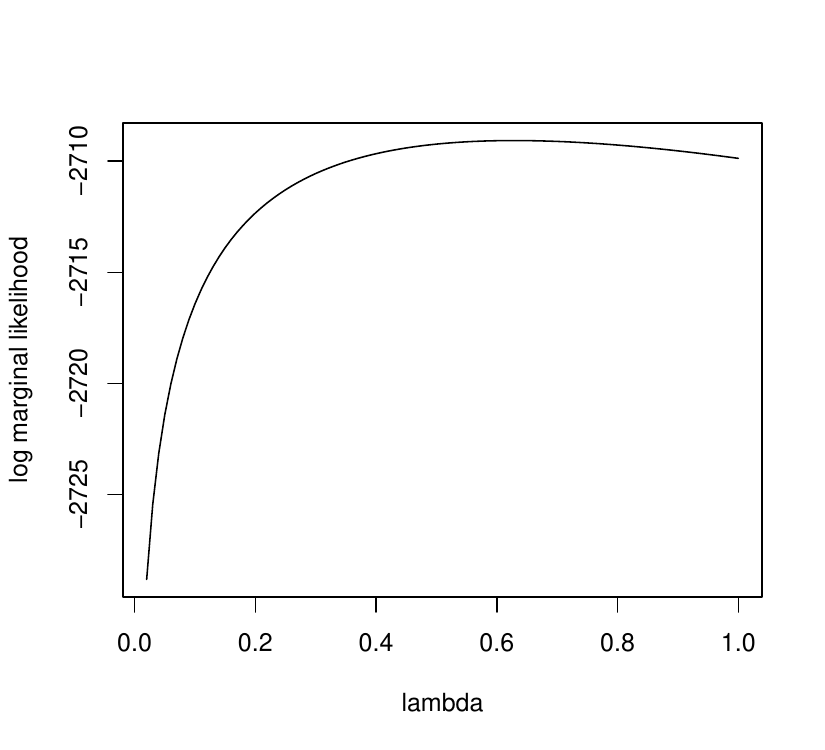}
    \includegraphics[width=0.49\linewidth]{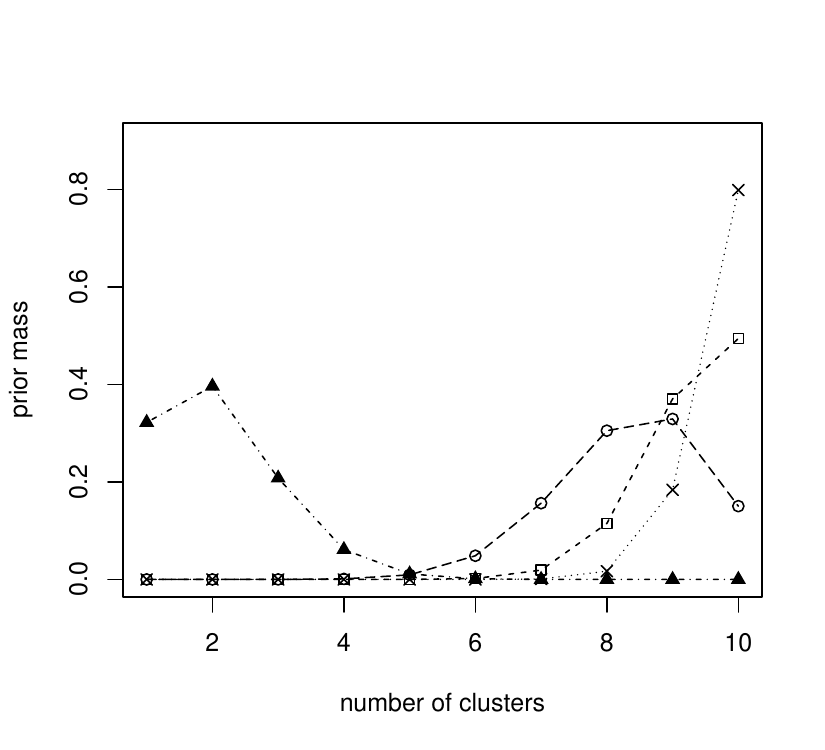}
    \caption{\small Log marginal likelihood (left-hand panel) and prior probability mass functions (right-hand panel) of the number $H_n$ of nonempty clusters in a sample ($n=400$) from an overfitted Gaussian location mixture ($K=10)$, with prior hyperparameters $\lambda=1$ (dot-dashed line with crosses), $\lambda= \hatl_n$=0.63 (dashed with squares); $\lambda=\hatl_n^{\text{res}}=0.367$ (long-dashed line with circles); and $\lambda=\lambda_n(\hat{K}_n)=0.0197$ (dot-dashed line with triangles). The true number of clusters is $K_0=2$. } 
\end{figure}
\end{ex}

\section{EBIB approximations to Bayesian inference} \label{sec:merging}

The results in the previous sections show that, even when the prior $\pi_\lambda$ can be degenerate (for limit values of $\lambda$) at all or some components of the true $\theta_0$, the MMLE {may} still converge to the prior-oracle $\lambda^*(\theta_0)$; however, for those $\theta_0$, under general conditions as in  Theorem 2 of \cite{p+r+s2014},  the EB posterior distribution of the vector $\theta$ {\em will not merge} with any Bayesian posterior from a non-degenerate prior. 
Strong merging of EB and Bayesian joint posterior distributions of $\theta$ can only hold in general in the regular case i.e. for true values $\theta_0$ for which $\sup_\lambda \pi_{\lambda}(\theta_0) < \infty$ and $\Lambda^* \subset \Lambda$; this is the case we consider in this section. In fact, as discussed in Section \ref{sec:intro}, even this case requires finer results: first-order comparisons do not allow for distinguishing among different posterior laws.  A simple but insightful example, helping intuition, is provided in Section D.3 of the Supplement. Indeed, under slightly stronger assumptions than those of Theorem  \ref{theo:MMLE}, the conclusions of Theorem 1 in \cite{p+r+s2014} hold in probability, so that the EB posterior distribution $ \Pi_{\hatl_n}(\cdot | Y) $ and, for any $\lambda$, the Bayesian posterior distribution $ \Pi_\lambda(\cdot | Y) $ are  not only consistent at $\theta_0$,  but their densities are also strongly merging, i.e. $\Vert\pi_{\hatl_n}(\cdot | Y) - \pi_{\lambda}(\cdot | Y)\Vert_1=o_{P_0}(1)$; in fact,  the EB posterior density merges strongly with {\em any}  Bayesian posterior from a prior density positive and continuous at $\theta_0$. Moreover, under additional  regularity assumptions on the likelihood (see e.g. \citealp{vaart_2000}), those Bayesian posteriors are asymptotically Gaussian, by the Bernstein-von Mises property; specifically, each of them merges in $L_1$ with the density sequence of $\mathcal{N}(\theta_0+ \Delta_n/\sqrt{n}, I(\theta_0)^{-1}/n)$, where $\Delta_n= I(\theta_0)^{-1}S_n(\theta_0)$,  $S_n(\theta_0)$ denotes the score process and $I(\theta_0)$ is the Fisher information at $\theta_0$. The same is true for the EB posterior density.  To understand which posterior distribution the EB posterior is specifically targeting,  it is thus necessary  to consider higher order approximations. To this end, {we here provide} a higher order $L_1$-comparison of posterior and predictive densities. For technical convenience, we assume that $\Lambda^*(\theta_0)$ consists of a single point $\lambda^* \equiv \lambda^*(\theta_0)$ such that $\pi_{\lambda^*}(\theta_0)<\infty$.

\subsection{Higher order analysis of posterior merging}\label{sec:nondegenerate}

Our higher order asymptotic comparison of posterior distributions is based on general conditions on the log likelihood (Conditions B.5--B.6 in Section B.3 of the Supplement) commonly required for higher order asymptotics, see for instance \cite{kass1990} on pages 483--484. In the sequel, for any function $g: \Theta \mapsto \mathbb{R}$, we write $\nabla g(\theta)$ and $\nabla^2 g(\theta)$  respectively for the vector of first derivatives and the matrix of second derivatives  with respect to $\theta$, and we set $\psi(\theta,\lambda)=\log \pi_\lambda(\theta)$, which is assumed to be differentiable (Condition B.7 in Section B.3 of the Supplement). Define $\Delta_{\theta_0}(\lambda,\lambda'): =	\nabla\psi(\theta_0, \lambda)- \nabla\psi(\theta_0, {\lambda}'),$ for two different hyperparameters $\lambda, \lambda' \in \Lambda$; moreover, let $I_0 $ be a short notation for $I(\theta_0)$.        
\begin{prop}\label{prop:TVpost_gen} Assume Conditions B.4(c)-(d) and B.6--B.7 in Sections B.2--B.3  of the Supplement are satisfied.  Then, for any $\lambda \in \Lambda$, 
	\begin{equation}\label{eq:quadratic}
	\Vert \pi_{\hat{\lambda}_n}(\cdot|Y)- \pi_{\lambda}(\cdot|Y)\Vert_1 =
	{\sqrt {\frac{2}{\pi}}}
	\left\lbrace
	\frac{
		\Delta_{\theta_0}(\hat{\lambda}_n,\lambda)^T I_0^{-1}
		\Delta_{\theta_0}(\hat{\lambda}_n,\lambda)}
	{n}\right\rbrace^{1/2} + o_{P_0}(n^{-1/2})
	\end{equation}
	and, in particular, $\Vert \pi_{\hat{\lambda}_n}(\cdot|Y)- \pi_{\lambda^*}(\cdot|Y)\Vert_1 =o_{P_0}(n^{-1/2})$. 
\end{prop}
The proof is presented in Section C.5 of the Supplement. It is based on a higher-order expansion of posterior densities, which remains valid for the EB posterior. From the proof it is clear that Proposition \ref{prop:TVpost_gen} extends to other types of estimators $\hat \lambda_n$, as soon as $\hatl_n $  converges in probability to $ \lambda^*(\theta_0)$. Interestingly, this is the case for the {\em proxy-MMLE} in Corollary \ref{cor:proxy}. We use such an extended result in Section \ref{sec:examples}, to treat variants of MMLE. 

Proposition \ref{prop:TVpost_gen} allows us to see that EBIB is indeed targeting the oracle posterior $\pi_{\lambda^*}(\cdot|y)$,  and that for all $\lambda $ such that $\Delta_{\theta_0}(\lambda^*, \lambda)>0$, the {Bayes} posterior will be much further apart from $\pi_{\lambda^*}(\cdot|y)$ or from $\pi_{\hat\lambda_n}(\cdot|y)$ since, for some $0<c_1<c_2<\infty$, we then have 
\begin{equation}\label{eq: post_comp2_gen}
		c_1 n^{-1/2} \leq 
		\Vert \pi_{\hat{\lambda}_n}(\cdot|y)- \pi_{\lambda}(\cdot|y)\Vert_1
		\leq c_2 n^{-1/2}, 
		\end{equation}
with $y$ in a set of $P_{\theta_0}^{(n)}$-probability tending to $1$. In practice, $\theta_0$ is unknown and so is the oracle value $\lambda^* = \lambda^*(\theta_0)$. Hence finding a value $\lambda$ such that, for all $\theta$, $\Delta_{\theta}( \lambda, \lambda^*(\theta)) = 0$ is often not possible a priori. In this regard EBIB with the MMLE is a way to ensure it.

 Proposition \ref{prop:TVpost_gen} is established for proper prior densities, using Conditions B.5--B.7 to guarantee that, for a neighbourhood $\hat{U}_n$ of the MLE, the ratio $\overline{m}_\lambda^{(n)}(Y;\hat{U}_n^c)/ \overline{m}_\lambda^{(n)}(Y;\hat{U}_n)$ is asymptotically negligible, where, for $U \subset \Theta$, $\overline{m}_\lambda^{(n)}(Y;U):= \int_{U}p_{\theta}^{(n)}(Y)/p_{\theta_0}^{(n)}(Y)\diff \Pi_\lambda(\theta)$. In general, as soon as the marginal likelihood is well defined, 
$$
\overline{m}_\lambda^{(n)}(Y;\hat{U}_n^c)= 		\frac{\Pi_\lambda(\hat{U}_n^c|Y)}{1-\Pi_\lambda(\hat{U}_n^c|Y)}
\overline{m}_\lambda^{(n)}(Y;\hat{U}_n).
$$
Thus, extensions to improper priors can be  easily obtained if the posterior mass outside $\hat{U}_n$ decays to $0$ sufficiently fast, as e.g. in linear regression with $g$-priors on the regression coefficients and improper prior on the noise variance  - illustrated in Section D.4.2 of the Supplement.	 To deal with data-dependent priors when the set $\Theta_n$ over which the likelihood is well-behaved   in  Condition B.5  differs from $\Theta$, we further require that for some $C>0$ the event  $P_{\theta_0}^{(n)}( \overline{m}_{\hat{\lambda}_n}^{(n)}(Y;\Theta_n^c)<Cn^{-H}) = 1 + o(1)$. We close with the following implications for Bayesian credible regions.  
\begin{rem}\label{rem:crediblelevel} An implication of Proposition \ref{prop:TVpost_gen} is that an EB credible region $C_n$ of level $1-\alpha$ for the parameter $\theta$, i.e. satisfying $\Pi_{\hatl_n}(C_n|Y)=1-\alpha$, has approximately the same credible level under the posterior distribution stemming from $\lambda^*\equiv \lambda^*(\theta_0)$, since Proposition \ref{prop:TVpost_gen} guarantees that $ \Pi_{\lambda^*}(C_n|Y)-(1-\alpha) = o_{P_0}(1/\sqrt{n})$.  At the same time, for any $\lambda $ such that $\Delta_{\theta_0}(\lambda, \lambda^*)\neq 0$, the rescaled discrepancy $(\Pi_{\lambda}(C_n|Y)-(1-\alpha))n^{1/2} $ remains positive with probability tending to $1$ (provided that $\alpha$ is small enough). Thus, informally, the EB credible region $C_n$ provides a more accurate approximation of an oracle Bayes credible region than of a non-oracle one.
\end{rem}

\subsection{Higher order analysis of predictive merging}\label{sec:nondegpred}

In many applications, prediction of future observations is the central goal, beyond parameter estimation.  Here we establish an analogue of Proposition \ref{prop:TVpost_gen} for merging of predictive distributions, considering observations of the form $Y=(Y_1, \ldots, Y_n)$, with $Y_i$'s all taking values in the same Euclidean subspace $\mathbb{Y}_0$, endowed with measure $\mu$.

In the Bayesian setting, prediction of a future observation $Y_{n+1}$ given $y=(y_1, \ldots, y_n)$, is solved by  the predictive distribution $ q_\lambda(y_{n+1} \mid y_{1:n}) = \int_{\Theta} p_\theta(y_{n+1} \mid y_{1:n}) \diff \Pi_\lambda(\theta|y_{1:n}), $ where $p_\theta(y_{n+1}\mid y_{1:n})$ is meant as the conditional density of $Y_{n+1}$ given the data {\em and} $\theta$. We focus on the case where the $Y_i$ are conditionally independent given $\theta$ but not necessarily identically distributed (that includes, in particular, fixed design regression); that is, $p_\theta^{(n)}(y_{1:n})= \prod_{i=1}^n p_{\theta, i}(y_i)$ for any $n\geq 1$. The EB predictive density using the MMLE $\hatl_n$ is $q_{\hatl_n}(y_{n+1} \mid y_{1:n})$,  and its merging properties are studied under the true density $p_0^{(n)}=\prod_{i=1}^{n}p_{\theta_0,i}$. As soon as the Bayesian and the EB posterior distributions  on $\theta$ are consistent, $q_{\hat\lambda_n}(\,\cdot\,|Y_{1:n})$ merges in $L_1$ with $q_{\lambda}(\,\cdot\,|Y_{1:n})$ for any $\lambda$, thus at first asymptotic order, no distinction between the EB and any Bayesian predictive distribution is possible. Note that this can be the case even when $\sup_\lambda \pi_\lambda(\theta_0)=\infty$.  Again, higher order refinements of merging results are only possible (in general) for the regular case where $\sup_\lambda\pi_\lambda(\theta_0)<\infty$, on which we focus.  Because, in a frequentist perspective, the $Y_i$ are independent, regularity conditions on the log likelihood, such as Condition B.5 in Section B.3 of the Supplement, can be written in a simpler form. Moreover, for merging of the predictive distributions, we need weaker assumption on the smoothness of the log likelihood (Condition B.8).  
\begin{prop}\label{prop: pred}
	Assume  that Conditions B.4(c)--(d) and B.6--B.8 in Sections B.2--B.3 of the Supplement  are verified. Then for any $\lambda \in \Lambda$, as $n \to \infty$
	\begin{equation}\label{eq:quadratic_pred}
	\Vert q_{\lambda}(\cdot | Y_{1:n})-			q_{\hat{\lambda}_n}(\cdot | Y_{1:n})\Vert_1 = \frac{1}{n}
	\int_\mathbb{Y}
	\left| 
	\Delta_{\theta_0}(\lambda,\hat{\lambda}_n)^T I_0^{-1}\nabla_\theta p_{\theta,n+1}(y_{n+1})\right|_{\theta=\theta_0} \diff y_{n+1}+o_{P_0}\left(\frac{1}{n}\right)
	\end{equation}
	and, in particular, $\Vert q_{\lambda^{*}}(\cdot\mid Y_{1:n})-q_{\hat{\lambda}_n}(\cdot \mid Y_{1:n})\Vert_1=o_{P_0}(n^{-1})$.
	\end{prop}
The proof is provided in Section C.6 of the Supplement. Note that in the simpler case of i.i.d. data, as soon as the function $y \mapsto \nabla p_{\theta_0}(y)$ is integrable and its components are continuous and linearly independent functions, then 
$$
\int_\mathbb{Y}
	\left| 
	\Delta_{\theta_0}(\lambda,\lambda^*)^T I_0^{-1}\nabla  p_{\theta_0}(y) \right| \diff \mu(y) > 0 \quad \Leftrightarrow \quad \Delta_{\theta_0}(\lambda,\lambda^*) \neq 0
	$$
and higher order predictive merging is equivalent to either $\Delta_{\theta_0}(\lambda,\lambda^*)=0$ for fixed $\lambda$, or $\Delta_{\theta_0}(\lambda,\lambda^*)=o(1)$ for any sequence $\lambda \equiv \lambda_n$. 

Finally, implications as in Remark \ref{rem:crediblelevel} hold for predictive intervals, too. If $R_n$ is an EB predictive region for $Y_{n+1}$ given $Y_{1:n}$, such that $\int_{R_n}  q_{\hatl_n}(y \mid Y_{1:n}) \, \diff\mu(y)=(1-\alpha)$, then  $n \; (\int_{R_n} q_{\lambda^*}(y \mid Y_{1:n}) \, \diff y - (1-\alpha))=o_{P_0}(1),$ while the discrepancy in predictive levels obtained with a non-oracle choice $\lambda \neq \lambda^*(\theta_0)$ would generally be of larger magnitude.

\subsection{Examples and extensions}\label{sec:examples}
We here illustrate merging, studied in its general aspects in the previous sections.  In fact, each of the following examples extends the reach of our theoretical results: in Example \ref{ex:finitegauss} (finite mixtures) the model is weakly identifiable; in Example \ref{ex: lasso}, the prior lacks smoothness at zero (exemplifying a typical behavior with hierarchical shrinkage priors); furthermore, in both examples, also merging of the {\em marginal} posterior distributions is formally inspected.  An additional example, dealing with improper priors and predictive merging in random design regression, is provided in Subsection D.4 of the Supplement.
\begin{ex}[{\em Gaussian finite mixtures}] \label{ex:finitegauss}  
Consider the Gaussian scale-location mixture model  $ Y_i |w, \mu_{1:K},v_{1:K}  \iidsim  \sum_{j=1}^K w_j \, \Norm(\mu_j, v_j)$, with a popular Normal Inverse-Gamma prior on the emission parameters  
\begin{equation} \label{eq:original_prior_dens}
 \mu_j \mid v_j \indsim \mathcal{N}(\xi, v_j/\tau),  \quad 
			v_j \iidsim \mathcal{IG}(\omega/2, \psi/2), \quad  w=(w_1,\ldots,w_K) \sim \mathcal{D}(1,...,1).
\end{equation}
as e.g. in  \cite{raftery1996}, Section 5.1. Let $\lambda=(\xi, \tau, \psi)$, while $\omega$ is a fixed hyperparameter. Data-dependent choices of the hyperparameters have been proposed {and commonly allowed in package implementations \citep[e.g., in the \texttt{R} package \texttt{AntMAN}, see][]{AntMAN2025}}, with the aim of specifying priors which are as weakly informative as possible \citep[see also][]{r+g97}; herein we consider the MMLE.
Differently from overfitted mixtures in Section \ref{sec:mixtures}, we here assume that the true number of components is known and equal to $K$, i.e. the true mixture weights are all positive and the true pairs of parameters $(\mu_{0,j}, v_{0,j})_{1\leq j\leq K}$ are distinct. In this case, for  each $\theta_0$ giving the true density $p_0^{(n)}=p_{\theta_0}^{(n)}$, there are  $K!-1$ permutations that  induce  the same data-generating density. As argued in Section D.5 of  the Supplement, this weak form of non-identifiability does not hamper the application of our results, provided that the true means $\mu_{0,j}$ are not all equal. If so, there still exists a unique oracle choice of the hyperparameter $\lambda^*(\theta_0)=(\xi^*, \tau^*, \psi^*)$, given by
			\begin{equation*}
			\begin{split}
			\xi^*= \sum_{j=1}^K \frac{1/v_{0,j}}{\sum_{l=1}^K 1/v_{0,j}}\mu_{0,j},
			\quad
			\tau^*=K \left\lbrace
			\sum_{j=1}^K
			\frac{(\mu_{0,j}-\xi^*)^2}{v_{0,j}}
			\right\rbrace^{-1},
			\quad
			\psi^*=K \omega 
			\left\lbrace
			\sum_{j=1}^K 1/v_{0,j}
			\right\rbrace^{-1},
			\end{split}
			\end{equation*}
and, under suitable assumptions on the set $\Lambda^{(n)}$ over which the MMLE  is computed, results on posterior and predictive merging, analogous to those in Proposition \ref{prop:TVpost_gen}--\ref{prop: pred}, hold. 

Again, exact evaluation of the MMLE can be lengthy, and EBIB via proxy-MMLEs can be more viable. 
{\scriptsize 
\begin{table}[h]
\centering
\begin{tabular}{c | cccc | cccc}
\hline
 & & {\small Mixture model 1} & {\small($K=3$) } &   & & {\small Mixture model 2} &  {\small($K=4$)} & \\
\hline 
{\small $n$ } & {\small100 } & {\small 600 } & {\small1100 } & {\small2100 } &  {\small100 } & {\small600} & {\small1100} & {\small2100}\\
 \hline
  {\small$\lambda_{1}$} 
  & {\small4.864}
  & {\small3.376}
  & {\small3.719}
 & {\small3.772} 
 & {\small5.639} & {\small 5.900} & {\small6.084} & {\small6.034}\\
  {\small$\lambda_2$}   & {\small0.061}
  & {\small0.125}
  & {\small0.104}
 &  {\small0.104}
 &  {\small0.012} & {\small0.014} & {\small0.014} & {\small0.014} \\
 {\small ${\lambda}_{3}$}   & {\small0.632}
  & {\small0.970}
  &{\small0.963}
 &{\small0.997}
 &{\small0.237} &{\small0.304} &{\small0.301} &{\small0.301} \\
 \hline
\end{tabular}
\caption{{\small Proxy-MMLE of $\lambda_j$, $j=1,2,3$, for increasing sample size, for the two mixture models in Example \ref{ex:finitegauss}. The oracle values $(\lambda_1^*, \lambda_2^*, \lambda_3^*)$ are, respectively, $(3.67, 0.12, 1)$ and $(6, 0.015, 0.3)$.}}
\label{fig:tableMMLE_mix}
\end{table}
}
\begin{figure}[h!]
    \centering
       \includegraphics[scale=1]{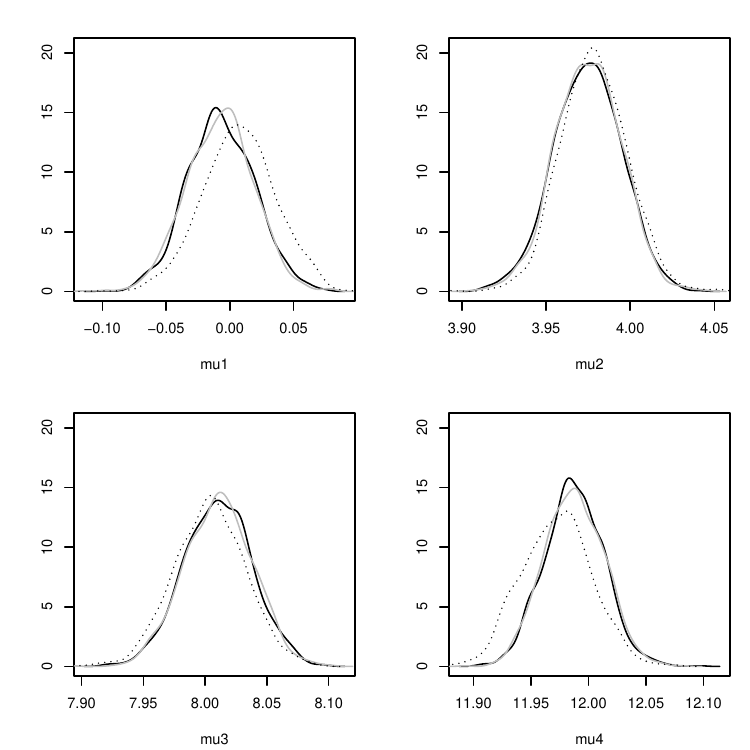}
    \caption{\small Marginal posterior densities {($n=2100$)} of the location parameters $\mu_1, \ldots,\mu_{4}$ for the Gaussian scale-location mixture model, with prior \eqref{eq:original_prior_dens}: EB marginal posterior densities with proxy-MMLE (black solid), Bayes with oracle hyperparameter $\lambda^*(\theta_0)$ (gray solid),  Bayes with $\lambda=(5,1,1)$ (dotted).} 
    \label{fig:regmix}
\end{figure}
Table \ref{fig:tableMMLE_mix} illustrates the convergence behavior of the proxy-MMLE to the oracle value $\lambda^*(\theta_0)$. The data are simulated from two mixture models:  (1) the mixture density $0.3 \mathcal{N}(0,1)+0.3 \mathcal{N}(4,1)+0.4 \mathcal{N}(7,1)$,  for which the oracle choice is $\lambda^*(\theta_0)=(3.\overline{6}, 0.12, 1)$, and (2) the mixture $0.2 \mathcal{N}(0,0.3)+0.4 \mathcal{N}(4,0.3)+0.2 \mathcal{N}(8,0.3)+0.2\mathcal{N}(12, 0.3)$, for which $\lambda^*(\theta_0)=(6,0.015, 0.3)$. The results reported in Table \ref{fig:tableMMLE_mix} refer to samples of increasing size $n=100, 600, 1100, 2100$. The proxy-MMLE is based on the maximum likelihood estimates of the model's parameters. As expected, at lower signal-to-noise the MLEs of $\theta$ are more accurate, in turn returning a proxy-MMLE closer to the oracle.

While our theoretical merging results apply to the joint posterior densities,  Figure \ref{fig:regmix} shows the {\em marginal} posterior densities of the location parameters of the four-component mixture model (2)  (in ascending order of magnitude), for $n=2100$ and three different choices of prior hyperparameters:   the EB marginal posterior densities with the proxy-MMLE, the oracle-Bayes, and the Bayes marginal posterior densities with fixed hyperparameter $\lambda=(5,1,1)$. In all cases, except for  $\mu_2$, the EB and the oracle-Bayes posterior densities are significantly closer to each other than to the Bayes posterior density with the given choice of $\lambda$. This suggests faster merging of the {\em marginal} EB and oracle posteriors.
\end{ex}

\begin{ex}[{\em Bayesian LASSO}] \label{ex: lasso}
Consider a linear regression model with no intercept 
$Y= X\beta + \varepsilon$, where $\varepsilon_i \mid \sigma^2 \iidsim \mathcal{N}(0,\sigma^2)$ and the design matrix $X$  satisfies $X^TX/n \to V$, with $V$  positive definite. Interest is in sparse models where some, but not all, true coefficients $\beta_{0,j}$ can be zero. A similar regression setting was considered in Example \ref{ex:mixedreg}, with independent $N(0, \tau_j^2)$ priors on the coefficients $\beta_j$ and an EB choice of the $\tau_j^2$; while a widely used Bayesian approach is based on continuous shrinkage priors with hierarchical construction 
\begin{align} \label{eq:shrinkage priors class} 
\beta_j \mid  \lambda, \tau_1^2, \ldots, \tau_d^2, \sigma^2  \indsim \Norm_d(0, v(\lambda,\sigma^2) \tau_j^2) \, \quad \mbox{with } \quad 
\tau_j^2  \mid \lambda  \iidsim h_\lambda, \quad j=1, \ldots,d  
\end{align}
where the function $v(\lambda,\sigma^2)$  controls global shrinkage and $\sigma^2\sim \pi$. As a prototype, we here consider the Laplace prior on $\beta$ proposed by \cite{parkCasella2008} for Bayesian LASSO 
\begin{equation}\label{eq:bayeslasso}
	\pi_\lambda(\beta \mid \sigma^2)= \prod_{j=1}^d \frac{\lambda}{2\sigma} e^{-\lambda |\beta_j|/\sigma},  
\end{equation} 
whose hierarchical representation has $v(\lambda,\sigma^2)=\sigma^2$ and $h_\lambda(\tau^2)= (\lambda^2/2)\exp(-\lambda^2\tau^2/2)$, an Exponential($\lambda^2/2$). Again, the choice of the hyperparameter $\lambda$, controlling shrinkage, is delicate and in an EBIB approach it could be selected by MMLE \citep{parkCasella2008}. The oracle value, that lies in the interior of $\Lambda$, is $\lambda^*(\beta_0, \sigma_0) = d \, \sigma_0/\sum_{j=1}^d |\beta_{0,j}|$  (the function $\lambda \mapsto \log \pi_\lambda(\beta_0|\sigma_0^2)$ is strictly concave and has a unique maximizer).  In this context, Proposition \ref{prop:TVpost_gen} is not directly applicable for comparisons,  since the derivatives $(\partial /\partial \beta_j)\log \pi_\lambda(\beta|\sigma_0^2)$ are not defined at $\beta_0$ whenever $\beta_{0,j}=0$. This singularity is a feature of the Laplace prior and more generally of other shrinkage priors, which are designed to push probability mass near zero. Note that for null $\beta_{0,j}$, the direct EB selection of $\tau^2_j$, as in Example \ref{ex:mixedreg}, leads to an EB posterior degenerate at zero; {although} the hierarchical Laplace prior $\pi_\lambda$ smooths this behaviorand we now have $\sup_\lambda \pi_\lambda(\theta_0) < \infty$, the degeneracy at the latent stage still has an effect - the hierarchical prior is not differentiable at those points.   Still,  assuming for simplicity that $X$ consists of orthogonal columns $x_{n,j}$,  $j=1, \ldots, d$, and that $\sigma^2$ is known, we can prove higher order merging for the joint {\em and} the marginal posterior densities of the regression coefficients $\beta_j$ for each $j$. More precisely, as  detailed in Section D.6.2 of the Supplement, the Bayesian marginal posterior density $\pi_{\lambda,j}(\,\cdot\, | Y)$ of $\beta_j$ and the EB marginal posterior density $\pi_{\hatl_n,j}(\,\cdot\,|Y)$ satisfy
\begin{align} \label{eq:shrinkagepriors}
|\hatl_n - \lambda|\frac{L_{n,j}}{\sqrt{n}} \leq 
\Vert \pi_{\hatl_n,j}(\,\cdot    \,|Y)-\pi_{\lambda,j}(\,\cdot     \,|Y) \Vert_1 \leq |\hatl_n - \lambda|\frac{U_{n,j}}{\sqrt{n}}
\end{align}
for some positive random sequences $L_{n,j}\equiv L_{n,j}(\lambda)$ and $U_{n,j}\equiv U_{n,j}(\lambda)$,  which are $O_{P_0}(1)$ and converge in distribution to positive random variables. Since, by arguments similar to those in Example \ref{ex:mixedreg}, Theorem \ref{theo:MMLE} applies  and $\hatl_n = \lambda^*(\theta_0)+o_{P_0}(1)$ (see Sections D.6.1-D.6.2 in the Supplement), we can conclude that faster merging holds {\em marginally} for each regression coefficient, whether its true value is non-null or zero: $\Vert \pi_{\hatl_n,j}(\,\cdot\,|Y)-\pi_{\lambda^*,j}(\,\cdot\,|Y) \Vert_1 =o_{P_0}(1/\sqrt{n})$, for $j=1, \ldots, d$. Furthermore, we can conclude that for any $\lambda>0$ there are positive sequences $L_n$ and $U_n$ (with same properties as above) such that the Bayesian and the EB {\em joint} posterior densities of the vector of regression coefficients satisfy 
\begin{align*} 
|\hatl_n - \lambda|
\frac{L_n}{\sqrt{n}}  \leq \Vert \pi_{\hatl_n}(\,\cdot\,|Y)-\pi_{\lambda}(\,\cdot\,|Y) \Vert_1   \leq
|\hatl_n - \lambda|
\frac{U_n}{\sqrt{n}}.
\end{align*}
 Once more, the distance of smaller asymptotic order is between EB and oracle-Bayes posteriors. This behaviour is illustrated in Figure \ref{fig:plot2}. Synthetic data are generated from a sparse model with $d=15$, a true $\beta$ vector with $11$ zeros and $(\beta_{0,12},\ldots \beta_{0,15})=(0.5, -2, 1, 3)$, and $\sigma^2=1$. The $X_i$ are iid Uniform$(-10,10)$. The prior of $\beta$ is as in \eqref{eq:bayeslasso}, with $\sigma^2 \sim\pi(\sigma) \propto 1/\sigma^2$. The oracle hyperparameter value is $\lambda^*(\theta_0)=2.3077$. 
\begin{figure}[h!]
    \centering
    \includegraphics[scale=1]{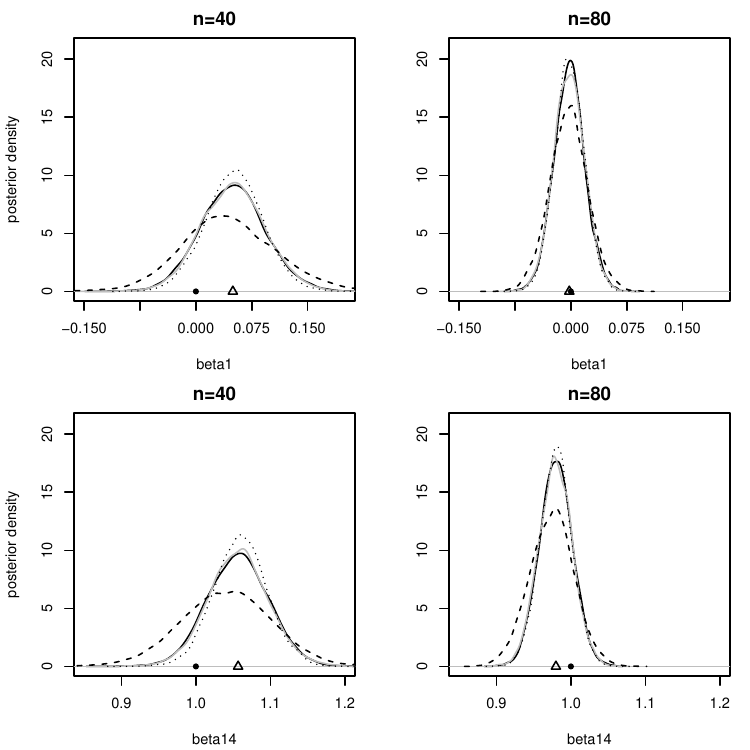}
    \caption{Bayesian LASSO. Posterior densities of $\beta_1$ and $\beta_{14}$: EB with MMLE (black solid), Bayes with oracle $\lambda^*$ (gray solid), Bayes with $\lambda=1$ (dotted) and $\lambda=8$ (dashed). True values $\beta_{0,j}$ are marked as black bullets, EB posterior means as empty triangles. } 
    \label{fig:plot2}
\end{figure}
\noindent Table \ref{fig:tableMMLE} shows the behavior of the MMLE $\hat{\lambda}_n$ for varying  $n$; as expected, it approaches the oracle value as $n$ increases. The MMLE is computed through the EM within Gibbs algorithm as in Park and Casella (2008). In fact, by an application of Corollary \ref{cor:proxy} (see Section D.6.1 in the Supplement) the MMLE $\hatl_n$ can be approximated by $ \hat{\lambda}^{\text{proxy}}_n = \lambda^*(\hat{\theta}_n)=  d \; \hat{\sigma}_n / \sum_{j=1}^d |\hat{\beta_j}_n|$, with $\hat{\theta}_n=(\hat \beta_n, \hat\sigma_n)$ the MLE of $\theta$. For a quick look,  both are reported in Table \ref{fig:tableMMLE}. 
{\scriptsize 
\begin{table}[h]
\centering
\begin{tabular}{c | cccc}
\hline
 {\small$n$} & {\small40} & {\small80} & {\small150} & {\small300}\\
 \hline
{\small MMLE $\hat{\lambda}_n$} & {\small2.503} & {\small2.026} & {\small2.207}  & {\small2.291} \\
 \hline
{\small Proxy-MMLE ${\lambda}^*(\hat{\theta}_n)$}
  & {\small2.679}
  & {\small2.104}
  & {\small2.274}
 & {\small2.336}\\
 \hline
\end{tabular}
\caption{\small MMLE $\hat{\lambda}_n$ approximated by EM withing Gibbs, and proxy-MMLE $\lambda^*(\hat{\theta}_n)$ }
\label{fig:tableMMLE}
\end{table}
}

Figure \ref{fig:plot2} illustrates the behavior of the marginal posterior densities.  Merging effects hold both for null and non-null regression coefficients; we report the results for $\beta_1$ and $\beta_{14}$, whose true values are $\beta_{0,1}=0$ and $\beta_{0,14}=1$.  The EB posterior density with MMLE (solid) is quite close to the oracle Bayes (solid, gray); for a comparison, we also plot the Bayesian posterior densities from a prior with $\lambda=8$ (dashed curve) and $\lambda=1$, fairly close to the oracle $\lambda^*(\beta_0,\sigma_0)=2.3$ (dotted).  For $n=40$ (left column), the differences are quite marked. For increasing $n$, differences are mitigated, as theoretically suggested by first order merging. As typical for higher-order asymptotic results, the differences would still be marked for large sample size if the model is higher-dimensional or more complex. 

Similar merging behavior can be envisaged for other continuous shrinkage priors of the form (\ref{eq:shrinkage priors class}), which as well may violate our smoothness assumptions (with special care needed in cases of asymptotes at sparse $\beta_{0}$'s for all $\lambda$, as for the  horseshoe), but a comprehensive scrutiny of the large variety of shrinkage priors goes beyond the scope of this paper.
\end{ex}

\section{Final remarks} \label{sec:final}

\cite{efron19} writes that \lq \lq Empirical Bayes has suffered of a philosophical identity problem", not being \lq \lq firmly attached to either frequentism or Bayesianism". Empirical Bayes {\em in Bayes} is  even more questionable. It is however popularly used, and in this paper we have dug into common beliefs to provide a cleaner interpretation. In non-degenerate parametric settings, and having chosen the family of priors, a Bayesian may regard EBIB as a close approximation of the posterior from a \lq \lq most informed" (oracle) prior, with computational advantages.  A frequentist can recognize it as a problem of maximum likelihood estimation in an unexplored case of model misspecification, for which we proved new results. 

Our findings hold for quite general settings;  still, more remains to be explored. Merging results for the EB marginal posterior distributions of regular components in partially degenerate cases are still unknown. We note that, in these cases, the Bayesian and the oracle-Bayes posterior laws will generally have different Bernstein-von Mises approximations so that merging would fail, even at first order; moreover, in (partially) degenerate cases the EB procedure is conjectured to be super-efficient. This can be future work.

We finally stress that, in the perspective here taken, the class of priors is considered as given. We do not discuss the more fundamental problem of choosing the family of priors. As seen, while first order asymptotics flattens the distinctions among different priors, even in comparing  EB and \lq fully Bayes' posterior distributions, our higher order results distinguish the EB posterior as a closer approximation of the Bayes-oracle posterior {\em within the chosen class of priors} $\{\pi_\lambda$, $\lambda \in \Lambda$\}. Still, a hierarchical (fully-Bayes) prior $\pi_H=\int \pi_\lambda(\theta) \diff H(\theta)$ expresses different and more structured information. Our finer results formally show (and confirm) that a comparison with this EB approach goes deeply beyond convergence to the most informed choice within a {\em given} class of priors.

\vspace{1em}

\paragraph{Supplementary Material.} 
The supplementary material document contains additional theoretical results, all the proofs of the main asymptotic results presented in this article and  details for the examples.

\paragraph{Acknowledgments.} 
The authors are sincerely grateful to the reviewers for many constructive and helpful comments.  The project leading to this work has received funding from the European Research Council (ERC) under the European Union’s Horizon 2020 research and innovation programme (grant agreement No. 834175). S.P. is supported by the European Union - Next Generation EU Funds, PRIN 2022 (2022CLTYP4).

\bibliographystyle{chicago}
\bibliography{ref}

\end{document}